\documentclass{amsart}
\usepackage{mathrsfs}
\usepackage{amscd, amssymb, amsmath, amsthm, amsfonts}
\usepackage{latexsym, graphics, graphicx, psfrag}
\usepackage{color,xcolor,colortbl}
\usepackage[all]{xy}

\newtheorem{thm}{Theorem}[section]

\newtheorem{lem}[thm]{Lemma}
\newtheorem{prop}[thm]{Proposition}
\newtheorem{que}[thm]{Question}
\theoremstyle{definition}

\theoremstyle{remark}

\usepackage[linkcolor=red, citecolor=blue]{hyperref}

\title[Virtual homological eigenvalue and mapping torus of pseudo-Anosov maps]{Virtual homological spectral radius and mapping torus of pseudo-Anosov maps}

\author[Hongbin~Sun]{Hongbin Sun}
\address{
    Department of Mathematics\\
    University of California, Berkeley\\
    Berkeley, CA 94720, USA}
\email{
    hongbins@math.berkeley.edu}

\thanks{The author is partially supported by NSF Grant No. DMS-1510383.}

\date{}

\begin{document}

\begin{abstract}
In this note, we show that, if a pseudo-Anosov map $\phi:S\to S$ admits a finite cover whose action on the first homology has spectral radius greater than $1$, then the monodromy of any fibered structure of any finite cover of the mapping torus $M_{\phi}$ has the same property.
\end{abstract}

\maketitle

\section{Introduction}

For a pseudo-Anosov map $\phi:S\to S$, among the data one can naturally associate to it, there are two real numbers that are very interesting to the author. One number is its dilatation, i.e. the stretch factor of the invariant singular foliation/lamination, the other one is the spectral radius of (the greatest modulus of eigenvalues of) $\phi_*:H_1(S;\mathbb{C})\to H_1(S;\mathbb{C})$.

For a pseudo-Anosov map $(S,\phi)$, by saying the {\it spectral radius} of the mapping class $(S,\phi)$, we mean the spectral radius of its homological action $\phi_*:H_1(S;\mathbb{C})\to H_1(S;\mathbb{C})$.

The spectral radius of a pseudo-Anosov map is always smaller or equal to the dilatation. Actually, even if we take the supremum of the spectral radii among all finite covers of the pseudo-Anosov map (the virtual spectral radius), it is still usually strictly smaller than the dilatation. More precisely, in \cite{McM3}, McMullen proved that the strict inequality holds if and only if the invariant singular foliation has an odd index singular point in the interior of the surface.

On the other hand, we do not know whether the virtual spectral radius of a pseudo-Anosov map is always greater than one. In particular, the following question is raised in \cite{Ko1}, where Koberda attributes this question to McMullen.

\begin{que}\label{eigenvalue}
Does any pseudo-Anosov map on a surface can be lifted to some finite cover, such that the spectral radius of the lifted map is greater than $1$?
\end{que}

If there exists a finite cover $(\tilde{S},\tilde{\phi})$ of $(S,\phi)$ such that the spectral radius of $\tilde{\phi}$ is greater than $1$, we say that $(S,\phi)$ has the {\it spectral lifting property}.

In \cite{Ha}, it is shown that, for any infinite order mapping class on a surface with boundary, it can be lifted to a finite cover, such that the lifted action on the first homology has infinite order. This is positive evidence for Question \ref{eigenvalue}.

Another important object associated to a mapping class $\phi:S\to S$ is its mapping torus $M_{\phi}$, i.e. $M_{\phi}=S\times I/(x,0)\sim (\phi(x),1)$. There have been a lot of works on studying mapping classes by using $3$-manifold topology. For example, in \cite{McM1} and \cite{McM2}, McMullen used the Alexander polynomial and Teichm\"{u}ller polynomial to study the monodromy maps of different fibered classes in a fixed fibered cone of a $3$-manifold. In \cite{Ko2}, Koberda related the spectral lifting property with the homology growth of finite covers of fibered $3$-manifolds.

In this note, we show that whether a pseudo-Anosov map has the spectral lifting property is actually a property of the commensurable class of its mapping torus, instead of being a property of the mapping class itself. It can be seen as positive evidence for Question \ref{eigenvalue}. In particular, we show the following theorem (some terminology will be defined in Section \ref{preliminary}).

\begin{thm}\label{main}
Let $\phi:S\to S$ be a pseudo-Anosov map and let $M_{\phi}$ be its mapping torus, then the following statements are equivalent to each other.
\begin{enumerate}
\item There exists a finite cover $(\tilde{S}, \tilde{\phi})$ of $(S,\phi)$, such that the spectral radius of $(\tilde{S},\tilde{\phi})$ is greater than $1$.
\item $M_{\phi}$ has a finite cover $N$, such that the Mahler measure of the multivariable Alexander polynomial of $N$ is greater than $1$, i.e. $\mathbb{M}(\Delta_N)>1$.
\item For any finite cover $N'$ of $M_{\phi}$, and any fibered structure of $N'$, the monodromy has a finite cover such that its spectral radius is greater than 1.
\end{enumerate}
\end{thm}

The proof of Theorem \ref{main} uses Mahler measures of (twisted) Alexander polynomials of $3$-manifolds. Actually, it is shown in \cite{SW, Le} that the statements in Theorem \ref{main} are also equivalent to a statement on homological torsion growth of finite abelian covers of $3$-manifolds, see Theorem \ref{torsiongrowth}.

After the resolution of the Virtual Haken and Virtual Fibered Conjectures (\cite{Ag1}), more and more evidence suggests that the family of all finite covers of a single hyperbolic $3$-manifold form a large class of $3$-manifolds, and it may satisfy many prescribed properties. In particular, Agol asked the following question about the monodromy of virtual fibered structures in \cite{Ag2}.

\begin{que}\label{orientable}
Does a finite volume hyperbolic $3$-manifold $M$ admit a finite cover which fibers over $S^1$ with orientable foliation of the pseudo-Anosov map?
\end{que}

For a pseudo-Anosov map with orientable invariant foliation, it is well known that its spectral radius is equal to its dilatation, which is greater than $1$. An immediate consequence of Theorem \ref{main} is that, a positive answer of Question \ref{orientable} implies a positive answer of Question \ref{eigenvalue}.

\subsection*{Acknowledgement}
The author is grateful to Ian Agol for many insightful conversations. The author thanks Yi Liu and Thomas Koberda for valuable communications on a previous version of this paper, and also thanks the anonymous referee for helpful comments.

\section{Preliminaries}\label{preliminary}

In this section, we first review some basic material of surface automorphisms and their finite covers, then review the definition and some properties of twisted Alexander polynomials and the Mahler measure.

\subsection{Surface Automorphisms}

Let $S$ be a compact orientable surface with $\chi(S)<0$ and let $\phi:S\to S$ be an orientation preserving self-homeomorphism (surface automorphism), then Thurston showed the following Nielsen--Thurston trichotomy of $(S,\phi)$ (see \cite{Th}).

\begin{thm}\label{nielsenthurston}
For any surface automorphism $(S,\phi)$, $\phi$ is isotopic to exactly one of the following three standard forms.
\begin{itemize}
\item $\phi$ is periodic.
\item $\phi$ is reducible, i.e. there is a finite collection of disjoint essential simple closed curves $\mathcal{C}$ on $S$ such that $\phi(\mathcal{C})=\mathcal{C}$.
\item $\phi$ is pseudo-Anosov, i.e. there are two transverse singular measured foliations that are invariant under the action of $\phi$, such that $\phi$ stretches one of the singular measured foliation and contracts the other one.
\end{itemize}
\end{thm}

Among these three types of surface automorphisms, the pseudo-Anosov type is the most interesting one. Apparently, the periodic type is relatively simple. For a reducible map, the surface can be decomposed to subsurfaces along $\mathcal{C}$, such that $\phi$ acts on each $\phi$-component of $S\setminus \mathcal{C}$ as a periodic or pseudo-Anosov automorphism.

Now we consider lifting surface automorphisms to finite covers.

For a surface automorphism $\phi:S\to S$ and a finite cover $\pi:\tilde{S}\to S$, we say that $(S,\phi)$ lifts to $\tilde{S}$ if there is an automorphism $\tilde{\phi}:\tilde{S}\to \tilde{S}$ such that $\pi\circ\tilde{\phi}=\phi\circ\pi$, and we call $(\tilde{S},\tilde{\phi})$ a finite cover of $(S,\phi)$.

Given a surface automorphism $\phi:S\to S$ and a finite cover $\pi:\tilde{S}\to S$, a lift $\tilde{\phi}:\tilde{S}\to \tilde{S}$ may or may not exist. Actually, an exercise in covering space theory implies that, a lift $\tilde{\phi}$ on $\tilde{S}$ exists if and only if $\phi_*:\pi_1(S)\to \pi_1(S)$ fixes the conjugacy class of $\pi_*(\pi_1(\tilde{S}))<\pi_1(S)$.

Since $\pi_1(S)$ is finitely generated, it has only finitely many conjugacy classes of subgroups with index $[\pi_1(S):\pi_1(\tilde{S})]$. So there exists a positive integer $k$, such that $(S,\phi^k)$ admits a lift to $\tilde{S}$ (there is a $k$ bounded by the number of conjugacy classes that works). It is also known that if $\pi$ is a characteristic finite cover, i.e. $\pi_1(\tilde{S})$ is a finite index characteristic subgroup of $\pi_1(S)$, a lift always exists.

On the other hand, for an automorphism $\phi:S\to S$ and a finite cover $\pi:\tilde{S}\to S$, $(S,\phi)$ may have different finite covers $(\tilde{S},\tilde{\phi})$ and $(\tilde{S},\tilde{\phi}')$, and they differ by a deck transformation of $\pi:\tilde{S}\to S$. In this case, there exists a positive integer $n$ (bounded by the order of the deck transformation group), such that $(\tilde{\phi})^n=(\tilde{\phi}')^n$.

By the definition of periodic, reducible and pseudo-Anosov mapping classes, it is easy to see that taking a lift of a mapping class (when it exists) does not change the Nielsen--Thurston type.

The following lemma concerns the behavior of the spectral radius when taking finite covers of surface automorphisms. Although the proof is simple, since the spectral radius of pseudo-Anosov maps is the main object of this paper, we give a proof here.

\begin{lem}\label{simple}
Let $(S,\phi)$ be a surface automorphism with spectral radius greater than $1$, then for any finite cover $(\tilde{S},\tilde{\phi})$ of $(S,\phi)$, the spectral radius of $\tilde{\phi}$ is also greater than $1$.
\end{lem}

\begin{proof}
Let $\pi:\tilde{S}\to S$ be the finite cover, and $K$ be the kernel of $\pi_*:H_1(\tilde{S};\mathbb{C})\to H_1(S;\mathbb{C})$. Since $\pi_*$ is surjective, we can identify $H_1(S;\mathbb{C})$ as a subspace of $H_1(\tilde{S};\mathbb{C})$, and get a decomposition $H_1(\tilde{S};\mathbb{C})=K\oplus H_1(S;\mathbb{C})$.

Since $(\tilde{S},\tilde{\phi})$ is a finite cover of $(S,\phi)$, $\pi_*\circ \tilde{\phi}_*=\phi_*\circ \pi_*$ holds. This implies that $K$ is a $\tilde{\phi}_*$-invariant subspace of $H_1(\tilde{S};\mathbb{C})$. Let $A$ be a squared matrix representing $(\tilde{\phi}_*)|_K:K\to K$, and let $C$ be a squared matrix representing $\phi_*:H_1(S;\mathbb{C})\to H_1(S;\mathbb{C})$. Then the lifted homological action $\tilde{\phi}_*:H_1(\tilde{S};\mathbb{C})\to H_1(\tilde{S};\mathbb{C})$ is represented by a matrix in the form of
$\begin{pmatrix}
A & B\\
0 & C
\end{pmatrix}$.

Since the spectral radius of $\phi$ is greater than $1$, the spectral radius of $C$ is greater than $1$. By simple linear algebra, the spectral radius of $\begin{pmatrix}
A & B\\
0 & C
\end{pmatrix}$, which equals the spectral radius of $\tilde{\phi}$, is also greater than $1$.
\end{proof}

\subsection{Twisted Alexander Polynomials of $3$-manifolds}
We will only state the definition of twisted Alexander polynomials in sufficient generality as to be convenient for our applications.

We will use $R$ to denote a domain. In this paper, we only consider the case that $R=\mathbb{C}$ or $\mathcal{O}_K$ (the ring of algebraic integers of a number field $K$).

For a compact manifold $M$ (or a finite CW-complex), let $\gamma:\pi_1(M)\to GL(n;R)$ be a group representation, and let $\psi:\pi_1(M)\to F$ be a homomorphism to a free abelian group $F$. Then $\gamma$ and $\psi$ induce a representation $$\alpha=\gamma\otimes \psi:\pi_1(M)\to GL(n;R)\otimes \mathbb{Z}[F]=GL(n;R[F]),$$ and it gives a $\pi_1(M)$-action on $R[F]^n$. In this paper, we will only use the $n=1$ case.

Let $\tilde{M}$ be the universal cover of $M$, then the (first) twisted Alexander module of $(M,\alpha)$ is the finitely generated $R[F]$-module $$H_1(\tilde{M};R[F]^n)=H_1(C_*(\tilde{M})\otimes_{\mathbb{Z}[\pi_1(M)]}R[F]^n).$$

For a finitely generated $R[F]$-module $A$, let $P$ be a presentation matrix of $A$ with size $r\times s$ and entries in $R[F]$. Then the order of $P$ is defined to be the greatest common divisor of the determinants of all $s\times s$ submatrices of $P$. The order is an element in $R[F]$ and well-defined up to multiplying a unit of $R[F]$. Then we define the twisted Alexander polynomial of $(M,\alpha)$ to be the order of $H_1(\tilde{M};R[F]^n)$, and we denote it by $\Delta_M^{\alpha}$.

When $\gamma:\pi_1(M)\to GL(1;\mathbb{Z})=\{\pm 1\}$ is the trivial representation, we get an ordinary Alexander polynomial of $M$ and denote it by $\Delta_M^{\psi}=\Delta_M^{1\otimes \psi}$. Then $\Delta_M^{\psi}$ is an element in $\mathbb{Z}[F]$. When $\gamma$ is trivial and $\psi=\pi:\pi_1(M)\to H_1(M;\mathbb{Z})/\text{Tor}$ is the standard homomorphism to the free part of $H_1(M;\mathbb{Z})$, we get the standard multivariable Alexander polynomial of $M$, which is denoted by $\Delta_M^{\pi}$, or simply $\Delta_M$. When $\gamma$ is trivial and $\psi$ corresponds to a cohomology class $a\in H^1(M;\mathbb{Z})$, i.e. $\psi$ equals the composition $\pi_1(M)\xrightarrow{\pi} H_1(M;\mathbb{Z})/\text{Tor}\xrightarrow{a}\mathbb{Z}$, we get the (single variable) Alexander polynomial of $M$ with respect to $a\in H^1(M;\mathbb{Z})$, which is denoted by $\Delta_M^a$.

Let $\xi:H_1(M;\mathbb{Z})/\text{Tor}\to S^1\subset GL(1;\mathbb{C})$ be a homomorphism, and let $\psi:\pi_1(M)\to F$ be a surjective homomorphism to a free abelian group $F$, then $\Delta_M^{\xi\otimes \psi}$ is a polynomial in $\mathbb{C}[F]$. For any surjective homomorphism $a:F \to \mathbb{Z}$, there is a relation between $a(\Delta_M^{\xi\otimes \psi})$ and $\Delta_M^{\xi\otimes (a\circ\psi)}$. The following lemma follows from Proposition 2, Proposition 3 and Proposition 5 of \cite{FV} and a simple computation.

\begin{prop}\label{restriction}
Let $M$ be a compact $3$-manifold with empty or tori boundary, and suppose that the rank of $F$ is at least $2$. Let the group ring $\mathbb{Z}[\mathbb{Z}]$ be identified with $\mathbb{Z}[t^{\pm 1}]$. Then we have $$\Delta_M^{\xi\otimes (a\circ\psi)}=a(\Delta_M^{\xi\otimes \psi})\cdot p(t).$$ Here $p(t)=1$ if the restriction of $\xi$ on $\ker{(a\circ\psi)}$ is nontrivial; $p(t)=t-\bar{\xi}(a^*)$ if the restriction of $\xi$ on $\ker{(a\circ\psi)}$ is trivial and $M$ is not a closed manifold, where $a^*\in H_1(M;\mathbb{Z})/\text{Tor}$ is an element such that $a\circ\psi(a^*)=1$; and $p(t)=(t-\xi(a^*))(t-\bar{\xi}(a^*))$ if the restriction of $\xi$ on $\ker{(a\circ\psi)}$ is trivial and $M$ is a closed manifold.

In particular, for the Mahler measure, $\mathbb{M}(\Delta_M^{\xi\otimes (a\circ\psi)})=\mathbb{M}(a(\Delta_M^{\xi\otimes \psi}))$ holds.
\end{prop}

We will define the Mahler measure of (multivariable) polynomials in the next subsection.

Let $p:\hat{M}\to M$ be a finite abelian cover with abelian deck transformation group $G$, then $\pi:\pi_1(M)\to H_1(M;\mathbb{Z})/\text{Tor}$ induces $$\hat{\pi}:\pi_1(\hat{M})\xrightarrow{p_*} \pi_1(M)\xrightarrow{\pi} H_1(M;\mathbb{Z})/\text{Tor}.$$ Then we can compute the Alexander polynomial $\Delta_{\hat{M}}^{\hat{\pi}}$ by twisted Alexander polynomials of $M$.

Let $\hat{G}$ be the set of all homomorphisms from $G$ to $S^1\subset GL(1;\mathbb{C})$, then it a finite set. For any $\xi\in \hat{G}$, we still use $\xi$ to denote the composition $$\pi_1(M)\xrightarrow{\pi} H_1(M;\mathbb{Z})/\text{Tor}\to G\xrightarrow{\xi} S^1\subset GL(1;\mathbb{C}).$$ Then we get a representation $$\xi\otimes\pi:\pi_1(M)\to GL(1;\mathbb{C}[H_1(M;\mathbb{Z})/\text{Tor}]).$$ Actually, the image of $\xi\otimes \pi$ lies in $GL(1;\mathcal{O}_K[H_1(M;\mathbb{Z})/\text{Tor}])$ for some cyclotomic field $K$ depending on $G$. Then the following formula is proved in \cite{DY}.

\begin{thm}\label{abeliancover}
Let $(t_1,t_2,\cdots,t_n)$ be a $\mathbb{Z}$-basis of $H_1(M;\mathbb{Z})/Tor$, then $$\Delta_{\hat{M}}^{\hat{\pi}}(t_1,\cdots,t_n)=\pm \Pi_{\xi \in \hat{G}}\Delta_M^{\xi\otimes \pi}(t_1,\cdots,t_n)=\pm \Pi_{\xi \in \hat{G}}\Delta_M^{\pi}(\bar{\xi}(t_1)t_1,\cdots,\bar{\xi}(t_n)t_n).$$ Here $\bar{\xi}$ denotes the complex conjugation of $\xi$.
\end{thm}

\subsection{The Mahler Measure}

For a polynomial $P(z)=a_0z^n+a_1z^{n-1}+\cdots+a_n\in \mathbb{C}[z]$ with $a_0\ne 0$, let $w_1,w_2,\cdots,w_n$ be the roots of $P(z)$ (with multiplicity). The Mahler Measure of $P(z)$ is defined to be $$\mathbb{M}(P)=|a_0|\cdot\Pi_{k=1}^n\max\{1,|w_k|\}.$$ By the Jensen's formula in complex analysis, the Mahler measure is equal to $$\mathbb{M}(P)=\exp\Big(\int_{S^1}\log{|P(z)|}\ \text{d}s \Big),$$ with $z=e^{2\pi is}$. We can take this integration formula as an alternative definition of the Mahler measure $\mathbb{M}(P)$.

The integration formula also gives the definition of Mahler measure for multivariable polynomials. For a polynomial $P({\bf z})=P(z_1,z_2,\cdots,z_m)\in\mathbb{C}[z_1,\cdots,z_m]$, its Mahler measure is defined by $$\mathbb{M}(P)=\exp\Big(\int_{T^m}\log{|P({\bf z})|}\ \text{d}{\bf s} \Big).$$ Here ${\bf z}=(z_1,z_2,\cdots,z_m)=(e^{2\pi is_1},e^{2\pi is_2},\cdots,e^{2\pi is_m})$.

For a multivariable polynomial with integer coefficients $P\in \mathbb{Z}[z_1,\cdots,z_m]$, $\mathbb{M}(P)\geq 1$ always holds. Moreover, it is also known exactly when $\mathbb{M}(P)=1$ holds (\cite{Bo}). Let $\Psi_n(z)$ be the $n$-th cyclotomic polynomial, an extended cyclotomic polynomial associated to $\Psi_n(z)$ is defined by the following data. For any positive integer $m$, and $m$ integers $v_1,v_2,\cdots,v_m$, let $b_k=\max\{0,-v_k\cdot\text{deg}(\Psi_n)\}$ for $k=1,2,\cdots m$, then $$z_1^{b_1}\dots z_m^{b_m}\Psi_n(z_1^{v_1}\dots z_m^{v_m})$$ is called an extended cyclotomic polynomial (with $m$ variables). Here we need the integers $b_k$ to make sure that each term of the extended cyclotomic polynomial has nonnegative power.

The following theorem characterizing $\mathbb{M}(P)=1$ is proved in \cite{Bo}.

\begin{thm}\label{cyclotomic}
For $P\in \mathbb{Z}[z_1,\cdots,z_m]$, $\mathbb{M}(P)=1$ if and only if $P$ is a product of monomials and extended cyclotomic polynomials.
\end{thm}

For a compact manifold $M$ (or a finite CW-complex), we fix a surjective homomorphism $\psi:\pi_1(M)\to \mathbb{Z}^n$. For any finite index subgroup $\Gamma<\mathbb{Z}^n$, let $M_{\Gamma}$ be the finite cover of $M$ corresponding to the subgroup $\psi^{-1}(\Gamma)<\pi_1(M)$. Then the Mahler measure of $\Delta_M^{\psi}$ is closely related with the homological torsion growth of $M_{\Gamma}$, when the minimal distance from nonzero elements in $\Gamma\subset \mathbb{Z}^n$ to zero goes to infinity (under any matric).

A special version of the following theorem is proved in \cite{SW}, and the general form is proved in \cite{Le}.
\begin{thm}\label{torsiongrowth}
For any compact manifold $M$ (or finite CW-complex), if $\Delta_M^{\psi}\ne 0$, then $$\limsup_{\langle\Gamma\rangle\to \infty}\frac{\log|\text{Tor}(H_1(M_{\Gamma};\mathbb{Z}))|}{|\mathbb{Z}^n/\Gamma|}=\mathbb{M}(\Delta_M^{\psi}).$$ Here $\langle\Gamma\rangle\to \infty$ means the minimal distance from nonzero elements in $\Gamma\subset \mathbb{Z}^n$ to zero goes to infinity.
\end{thm}

Actually, \cite{Le} also dealt with the case of $\Delta_M^{\psi}= 0$. In this case, the growth of homological torsion is equal to the Mahler measure of the first nonzero higher order Alexander polynomial. However, we do not need this stronger result in this paper.

\section{Proof of Theorem \ref{main}}

To prove Theorem \ref{main}, we first prove the following two propositions. These two propositions give relations between the Mahler measure of the Alexander polynomial of a $3$-manifold and virtual spectral radii of its monodromy maps.

\begin{prop}\label{localtoglobal}
Let $N$ be an orientable fibered hyperbolic $3$-manifold of finite volume. Suppose that for some fibered structure of $N$, the spectral radius of its monodromy is greater than 1, then $\mathbb{M}(\Delta_N^{\pi})>1$.
\end{prop}

\begin{proof}
Suppose that the fibered structure mentioned in the assumption corresponds to $a\in H^1(N;\mathbb{Z})$, and we denote its monodromy by $\psi:\Sigma\to \Sigma$. Then the Alexander polynomial $\Delta_N^a$ is equal to the characteristic polynomial of $\psi_*:H_1(\Sigma;\mathbb{Z})\to H_1(\Sigma;\mathbb{Z})$. So $\Delta_N^a$ is a polynomial in $\mathbb{Z}[t]$, with leading coefficient $1$. The assumption of this proposition implies that $\mathbb{M}(\Delta_N^a)>1$.

If $b_1(N)=1$, then $\Delta_N^{\pi}$ equals $\Delta_N^a$, so $\mathbb{M}(\Delta_N^{\pi})>1$. Now we assume that $b_1(N)\geq 2$.

We suppose that $\mathbb{M}(\Delta_N^{\pi})=1$, and get a contradiction. Since $\Delta_N^{\pi}$ has integer coefficients, Theorem \ref{cyclotomic} implies that $\Delta_N^{\pi}$ is a product of monomials and extended cyclotomic polynomials. For the cohomology class $a\in H^1(N;\mathbb{Z})$ and any extended cyclotomic polynomial $P\in \mathbb{Z}[H_1(N;\mathbb{Z})/\text{Tor}]$, it is easy to see that $a(P)=t^b\Psi_n(t^v)$ for some integers $v,b,n$. So $a(P)$ either equals a product of $t^b$ and a few cyclotomic polynomials, or equals $kt^b$ for some $k\in \mathbb{Z}$ (when $v=0$). This implies that $a(\Delta_N^{\pi})$ is a product of integers, monomials and cyclotomic polynomials.

By Proposition \ref{restriction}, and the fact that $\Delta_N^a$ has leading coefficient $1$, $a(\Delta_N^{\pi})=\frac{\Delta_N^a}{p(t)}$ is a nonzero polynomial with leading coefficient $1$. So $a(\Delta_N^{\pi})$ is only a product of monomials and cyclotomic polynomials, which has Mahler measure $\mathbb{M}(a(\Delta_N^{\pi}))=1$. However, this contradicts Proposition \ref{restriction}, which claims $$\mathbb{M}(a(\Delta_N^{\pi}))=\mathbb{M}(\Delta_N^a)>1.$$

This finishes the proof of $\mathbb{M}(\Delta_N^{\pi})>1$.
\end{proof}

\begin{prop}\label{globaltolocal}
Let $N$ be an orientable fibered hyperbolic $3$-manifold of finite volume, and suppose that $\mathbb{M}(\Delta_N^{\pi})>1$. Then for any fibered structure of $N$, the corresponding monodromy has the spectral lifting property.
\end{prop}

\begin{proof}
Take an arbitrary fibered structure of $N$, which corresponds to a cohomology class $b_1\in H^1(N;\mathbb{Z})$, and let the corresponding monodromy be $\psi:\Sigma\to \Sigma$. Take a $\mathbb{Z}$-basis $(b_1,b_2,\cdots,b_n)$ of $H^1(N;\mathbb{Z})$ containing $b_1$, and denote the dual basis of $H_1(N;\mathbb{Z})/\text{Tor}$ by $(t_1,t_2,\cdots,t_n)$.

Since $\mathbb{M}(\Delta_N^{\pi})>1$, by the definition of the Mahler Measure, we have $$\int_{T^n}\log{|\Delta_N^{\pi}({\bf z})|}\ \text{d}{\bf s}>0.$$ Here $z_k=e^{2\pi is_k}$, and $(z_1,z_2,\cdots,z_n)$ corresponds to the basis $(t_1,t_2,\cdots,t_n)$ of $H_1(N;\mathbb{Z})/\text{Tor}$.

We first integrate along the first coordinate and get $$\int_{T^{n-1}}\Big(\int_{S^1}\log{\big|\Delta_N^{\pi}(e^{2\pi i s_1},e^{2\pi i s_2},\cdots,e^{2\pi i s_n})\big|}\ \text{d}s_1\Big)\text{d}s_2\cdots\text{d}s_n>0.$$ So there exists a point $(r_2,\cdots,r_n)\in I^{n-1}$ such that
\begin{equation}\label{positive}
\int_{S^1}\log{\big|\Delta_N^{\pi}(e^{2\pi i s_1},e^{2\pi i r_2},\cdots,e^{2\pi i r_n})\big|}\ \text{d}s_1>0.
\end{equation}
By continuity, we can suppose that all of $r_2,\cdots,r_n$ are rational numbers.

Let $k$ be the least common multiple of the denominators of $r_2,\cdots,r_n$. We consider the finite abelian cover $p:\hat{N}\to N$ such that $\pi_1(\hat{N})$ is the kernel of the homomorphism $$\pi_1(N)\xrightarrow{\pi} H_1(N;\mathbb{Z})/\text{Tor}\cong \mathbb{Z}^n\xrightarrow{f} T^{n-1}.$$ Here the basis $(t_1,t_2,\cdots,t_n)$ of $H_1(N;\mathbb{Z})/\text{Tor}$ is identified with the standard basis of $\mathbb{Z}^n$, and the homomorphism $f:\mathbb{Z}^n\to T^{n-1}$ is defined by $$(k_1,k_2,\cdots,k_n)\to (e^{2\pi i\frac{k_2}{k}},\cdots,e^{2\pi i\frac{k_n}{k}}).$$ The deck transformation group $G$ of $\hat{N}\to N$ is isomorphic to $\mathbb{Z}_k^{n-1}$.

Now we consider a representation $\xi_0:H_1(N;\mathbb{Z})/\text{Tor}\to S^1$ defined by $\xi_0(t_1)=1$ and $\xi_0(t_i)=e^{-2\pi i r_i}$ for $i=2,\cdots,n$. Then $\xi_0$ vanishes on the kernel of $$H_1(N;\mathbb{Z})/\text{Tor}\cong\mathbb{Z}^n\xrightarrow{f}T^{n-1}.$$ So $\xi_0$ induces a representation $G\to S^1$, and we still denote this element in $\hat{G}$ by $\xi_0$.

Now we have two homomorphisms $\pi:\pi_1(N)\to H_1(N;\mathbb{Z})/\text{Tor}$ and $$\hat{\pi}:\pi_1(\hat{N})\xrightarrow{p_*} \pi_1(N)\xrightarrow{\pi} H_1(N;\mathbb{Z})/\text{Tor}.$$ We consider $b_1\in H^1(N;\mathbb{Z})$ as a homomorphism from $H_1(N;\mathbb{Z})/\text{Tor}$ to $\mathbb{Z}$, then it induces a homomorphism $$\mathbb{Z}[t_1^{\pm},\cdots,t_n^{\pm}]\cong \mathbb{Z}[H_1(N;\mathbb{Z})/\text{Tor}]\to \mathbb{Z}[\mathbb{Z}]\cong \mathbb{Z}[t^{\pm}],$$ which maps $t_1$ to $t$, and maps $t_2,\cdots,t_n$ to $1$.

By applying $b_1$ to the equation in Theorem \ref{abeliancover} and using $\xi(t_1)=1$ for any $\xi\in \hat{G}$, we get $$b_1\big(\Delta_{\hat{N}}^{\hat{\pi}}(t_1,t_2,\cdots,t_n)\big)=\pm\Pi_{\xi\in \hat{G}}b_1\big(\Delta_N^{\pi}(t_1,\bar{\xi}(t_2)t_2,\cdots,\bar{\xi}(t_n)t_n)\big)=\pm\Pi_{\xi\in \hat{G}}\Delta_N^{\pi}(t,\bar{\xi}(t_2),\cdots,\bar{\xi}(t_n)).$$

Then Proposition \ref{restriction} implies
\begin{equation}\label{product}
\mathbb{M}\Big(\Delta_{\hat{N}}^{p^*(b_1)}(t)\Big)=\mathbb{M}\Big(\Delta_{\hat{N}}^{b_1\circ \hat{\pi}}(t)\Big)=\mathbb{M}\Big(b_1\big(\Delta_{\hat{N}}^{\hat{\pi}}(t_1,t_2,\cdots,t_n)\big)\Big)=\Pi_{\xi\in\hat{G}}\mathbb{M}\Big(\Delta_N^{\pi}(t,\bar{\xi}(t_2),\cdots,\bar{\xi}(t_n))\Big).
\end{equation}

Since $b_1\in H^1(N;\mathbb{Z})$ is a fibered class, for any $\xi\in\hat{G}$, the polynomial $$\Delta_N^{\pi}(t,\bar{\xi}(t_2),\cdots,\bar{\xi}(t_n))=b_1\big(\Delta_{\hat{N}}^{\hat{\pi}}(t_1,\bar{\xi}(t_2)t_2,\cdots,\bar{\xi}(t_2)t_n)\big)$$ has leading coefficient $1$. So each Mahler measure in the product on the right hand side of equation (\ref{product}) is greater or equal to $1$. Moreover, since $\xi_0$ lies in $\hat{G}$, equation (\ref{positive}) implies that the term in the product of equation (\ref{product}) corresponding to $\xi_0$ is greater than $1$. So we have $$\mathbb{M}(\Delta_{\hat{N}}^{p^*(b_1)}(t))>1.$$

Since $b_1$ is a fibered class of $N$, $p^*(b_1)$ is a fibered class of $\hat{N}$, and let the corresponding monodromy be $\psi':\Sigma'\to \Sigma'$. Here $\Sigma'$ might be a disconnected surface, since $p^*(b_1)$ may not be a primitive class. Since $p^*(b_1)$ is a fibered class, $\Delta_{\hat{N}}^{p^*(b_1)}(t)$ is the characteristic polynomial of $\psi'_*:H_1(\Sigma';\mathbb{Z})\to H_1(\Sigma';\mathbb{Z})$. Then $\mathbb{M}(\Delta_{\hat{N}}^{p^*(b_1)}(t))>1$ implies the spectral radius of $\psi'$ is greater than $1$. Then the restriction of a power of $\psi'$ ($l$-th power) on a connected subsurface $\hat{\Sigma}\subset \Sigma'$ is a pseudo-Anosov map on a connected surface, and we denote this map by $\hat{\psi}:\hat{\Sigma}\to \hat{\Sigma}$. Then $\hat{\psi}:\hat{\Sigma}\to \hat{\Sigma}$ is a lift of $\psi^l:\Sigma\to \Sigma$, and its spectral radius is greater than $1$.

Let $\tilde{\Sigma}\to \Sigma$ be a finite characteristic cover factoring through $\hat{\Sigma}\to \Sigma$. Then $\psi:\Sigma\to \Sigma$ lifts to $\tilde{\psi}:\tilde{\Sigma}\to \tilde{\Sigma}$, and there exists a positive integer $k$ such that $\hat{\psi}^k$ lifts to $\bar{\psi}:\tilde{\Sigma}\to \tilde{\Sigma}$.

Since the spectral radius of $\hat{\psi}^k$ is greater than $1$, Lemma \ref{simple} implies that the spectral radius of $\bar{\psi}:\tilde{\Sigma}\to \tilde{\Sigma}$ is also greater than $1$. Since $\tilde{\psi}:\tilde{\Sigma}\to \tilde{\Sigma}$ and $\bar{\psi}:\tilde{\Sigma}\to \tilde{\Sigma}$ are lifts of $\psi:\Sigma\to \Sigma$ and $\psi^{kl}:\Sigma\to \Sigma$ respectively, $\tilde{\psi}$ and $\bar{\psi}$ have some common power. So the spectral radius of $\tilde{\psi}:\tilde{\Sigma}\to \tilde{\Sigma}$ is greater than $1$, which is the desired finite cover of $\psi:\Sigma\to \Sigma$.
\end{proof}

Now we are ready to prove Theorem \ref{main}.

\begin{proof}
(1) implies (2): We consider the fibered manifold $N=M_{\tilde{\phi}}=\tilde{S}\times I/(x,0)\sim (\tilde{\phi}(x),1)$. Then $N$ is a finite cover of $M_{\phi}$, and $\tilde{\phi}:\tilde{S}\to \tilde{S}$ is the monodromy of a fibered structure of $N$, such that its spectral radius is greater than $1$. Then Proposition \ref{localtoglobal} implies that $\mathbb{M}(\Delta_N^{\pi})>1$.

\bigskip

(2) implies (3): Take an arbitrary fibered class $b\in H^1(N;\mathbb{Z})$ of $N$, then Proposition \ref{globaltolocal} implies that the corresponding monodromy has the spectral lifting property. Let the mapping torus corresponding to this finite cover of monodromy be denoted by $\hat{N}$. Then $\hat{N}$ is a finite cover of $N$, and the spectral radius of the monodromy of this fibered structure is greater than $1$.

Take a common finite cover of $\hat{N}$ and $N'$, and denoted it by $\tilde{N}$. Then the above fibered structure of $\hat{N}$ lifts to a fibered structure of $\tilde{N}$, and the spectral radius of the corresponding monodromy is still greater than $1$. Then Proposition \ref{localtoglobal} implies that $\mathbb{M}(\Delta_{\tilde{N}})>1$.

By applying Proposition \ref{globaltolocal} again, we know that the monodromy of any fibered structure of $\tilde{N}$ has the spectral lifting property. Since any fibered structure of $N'$ lifts to a fibered structure of $\tilde{N}$, the argument at the end of the proof of Proposition \ref{globaltolocal} implies that the monodromy of any fibered structure of $N'$ has the spectral lifting property.

\bigskip

(3) implies (1): This is trivially true, since $M_{\phi}$ is a finite cover ($1$-sheet cover) of itself.
\end{proof}

\bibliographystyle{amsalpha}

\end{document}